\documentclass[12pt]{amsart}
\usepackage{amssymb,mathrsfs}

\begin{document}

\title{Stable automorphic forms for the general linear group  }

\author{Jae-Hyun Yang}

\address{Yang Institute for Advanced Study
\newline\indent
Hyundai 41 Tower, No. 1905
\newline\indent
293 Mokdongdong-ro, Yangcheon-gu
\newline\indent
Seoul 07997, Korea}

\address{Department of Mathematics
\newline\indent
Inha University
\newline\indent
Incheon 22212, Korea}

\email{jhyang@inha.ac.kr\ \ or\ \ jhyang8357@gmail.com}

\newtheorem{theorem}{Theorem}[section]
\newtheorem{lemma}{Lemma}[section]
\newtheorem{proposition}{Proposition}[section]
\newtheorem{remark}{Remark}[section]
\newtheorem{definition}{Definition}[section]

\renewcommand{\theequation}{\thesection.\arabic{equation}}
\renewcommand{\thetheorem}{\thesection.\arabic{theorem}}
\renewcommand{\thelemma}{\thesection.\arabic{lemma}}
\newcommand{\bbr}{\mathbb R}
\newcommand{\bbs}{\mathbb S}
\newcommand{\bn}{\bf n}
\newcommand\charf {\mbox{{\text 1}\kern-.24em {\text l}}}
\newcommand\fg{{\mathfrak g}}
\newcommand\fk{{\mathfrak k}}
\newcommand\fp{{\mathfrak p}}
\newcommand\g{\gamma}
\newcommand\G{\Gamma}
\newcommand\ka{\kappa}
\newcommand\al{\alpha}
\newcommand\be{\beta}
\newcommand\lrt{\longrightarrow}
\newcommand\s{\sigma}
\newcommand\ba{\backslash}
\newcommand\lmt{\longmapsto}
\newcommand\CP{{\mathscr P}_n}
\newcommand\CM{{\mathcal M}}
\newcommand\BC{\mathbb C}
\newcommand\BZ{\mathbb Z}
\newcommand\BR{\Bbb R}
\newcommand\BQ{\mathbb Q}
\newcommand\Rmn{{\mathbb R}^{(m,n)}}
\newcommand\PR{{\mathcal P}_n\times {\mathbb R}^{(m,n)}}
\newcommand\Gnm{GL_{n,m}}
\newcommand\Gnz{GL_{n,m}({\mathbb Z})}
\newcommand\Gjnm{Sp_{n,m}}
\newcommand\Gnml{GL(n,{\mathbb R})\ltimes {\mathbb R}^{(m,n)}}
\newcommand\Snm{SL_{n,m}}
\newcommand\Snz{SL_{n,m}({\mathbb Z})}
\newcommand\Snml{SL(n,{\mathbb R})\ltimes {\mathbb R}^{(m,n)}}
\newcommand\Snzl{SL(n,{\mathbb Z})\ltimes {\mathbb Z}^{(m,n)}}
\newcommand\la{\lambda}
\newcommand\GZ{GL(n,{\mathbb Z})\ltimes {\mathbb Z}^{(m,n)}}
\newcommand\DPR{{\mathbb D}(\PR)}
\newcommand\Rnn{{\mathbb R}^{(n,n)}}
\newcommand\Yd{{{\partial}\over {\partial Y}}}
\newcommand\Vd{{{\partial}\over {\partial V}}}
\newcommand\Ys{Y^{\ast}}
\newcommand\Vs{V^{\ast}}
\newcommand\DGR{{\mathbb D}(\Gnm)}
\newcommand\DKR{{\mathbb D}_K(\Gnm)}
\newcommand\DKS{{\mathbb D}_{K_0}(\Snm)}
\newcommand\fa{{\frak a}}
\newcommand\fac{{\frak a}_c^{\ast}}
\newcommand\SPR{S{\mathcal P}_n\times \Rmn}
\newcommand\DSPR{{\mathbb D}(\SPR)}
\newcommand\BD{{\mathbb D}}
\newcommand\SP{{\mathfrak P}_n}

\thanks{
\endgraf Keywords and phrases\,: polarized real tori, automorphic forms, stability of automorphic forms,\\
\indent the Jacobian real locus.}

\begin{abstract} In this paper, we introduce the notion of the stability of automorphic forms for the general linear group and relate the stability of automorphic forms to the moduli space of real tori and the Jacobian real locus.
\end{abstract}
\maketitle

\vskip 7mm

\centerline{\large \bf Table of Contents}

\vskip 0.75cm $ \quad\qquad\textsf{\large \ 1.
Introduction}$\vskip 0.021cm

$\quad\qquad \textsf{\large\ 2. Review on the geometry of $GL(n,\BR)/O(n,\BR)$}$
\vskip 0.0421cm

$ \quad\qquad  \textsf{\large\ 3. Polarized real tori }$
\vskip 0.0421cm

$ \quad\qquad \textsf{\large\ 4. The moduli of polarized real tori}$
\vskip 0.0421cm

$ \quad\qquad  \textsf{\large\ 5. Automorphic forms for $GL(n,\BR)$ }$
\vskip 0.0421cm

$ \quad\qquad
\textsf{\large\ 6. Stable automorphic forms for the general linear group}$
\vskip 0.0421cm

$ \quad\qquad\ \textsf{\large References }$

\vskip 10mm


\begin{section}{{\bf Introduction}}
\setcounter{equation}{0}

We let
$$\CP =\left\{\, Y\in \BR^{(n,n)}\,\vert\ Y=\,^tY>0\ \right\}$$ be
the open convex cone of positive definite symmetric real matrices of degree $n$
in the Euclidean space $\BR^{n(n+1)/2},$ where $F^{(k,l)}$ denotes
the set of all $k\times l$ matrices with entries in a commutative
ring $F$ for two positive integers $k$ and $l$ and $^t\!M$ denotes
the transpose of a matrix $M$. Then the general linear
group $GL(n,\BR)$ acts on $\CP$ transtively by
\begin{equation}
g\cdot Y=gY\,^tg,\ \ \ \ g\in GL(n,\BR),\ Y\in \CP.
\end{equation}
Therefore $\CP$ is a symmetric space which is diffeomorphic to the
quotient space $GL(n,\BR)/O(n,\BR)$, where $O(n,\BR)$ denotes the
real orthogonal group of degree $n$. Atle Selberg \cite{S1} investigated
differential operators on $\CP$ invariant under the action (1.1)
of $GL(n,\BR)$ (cf. \cite{M1, M2}). Using these invariant differential operators on $\CP$,
automorpic forms for $GL(n,\BR)$ were investigated thereafter (cf. \cite{B, G, Gr2, IT, T}).
The Siegel $\Phi$ operator plays an important role in the theory of Siegel modular forms
(cf.\, \cite{CS, F1, F2, M2}). Douglas Grenier \cite{Gr2} constructed an analogue of the
Siegel $\Phi$ operator called the {\sf Grenier operator} for automorphic forms for $GL(n,\BR)$. The Grenier operator is applied to study the Maass-Selberg relation for $GL(n,\BR)$.

\vskip 2mm
The goal of this article is to introduce the notion of stable automorphic forms for the general linear
group using the Grenier operator and relate the stability of automorphic forms to
the study of the moduli space of polarized real tori and the Jacobian real locus. This paper is organized as follows. In section 2,
we briefly review the geometry of the symmetric space $\CP=GL(n,\BR)/O(n,\BR)$ and spherical functions on
$\CP$. In section 3, we review some results on real polarized abelian varieties and then recall the notion of polarized real tori introduced by the author \cite{Y3}. In section 4, we roughly outline the moduli space of polarized real tori and the Jacobian real locus. In section 5, we review  the Fourier expansion of an automorphic form for $GL(n,\BR)$ and the Satake compactification of $GL(n,\BZ)\backslash \CP$ obtained by Grenier \cite{Gr1, Gr2, Gr3}. In the final section, we introduce the notion of stable automorphic forms for the general linear group using the Grenier operator and relate the stability of automorphic forms for $GL(\infty)$ to the study of the moduli space of polarized real tori and the Jacobian real locus. We also give an example of a stable automorphic form for $GL(\infty)$.
We prove that $(E_n(\alpha_n,Y)\,\vert\,n\geq 1)$ is a stable automorphic form
for $\Gamma_\infty$, where $E_n(\alpha_n,Y)$ denotes the Selberg Eisenstein series defined by the formula (5.5). See Theorem 6.2 for the precise statement.
This subject adds a new area to the theory of automorphic forms for the general linear group.

\vskip 0.31cm \noindent {\bf Notations:} \ \ We denote by
$\BQ,\,\BR$ and $\BC$ the field of rational numbers, the field of
real numbers and the field of complex numbers respectively. We
denote by $\BZ$ and $\BZ^+$ the ring of integers and the set of
all positive integers respectively. $\BR^{\times}$ (resp. $\BC^{\times}$)
denotes the group of nonzero real (resp. complex) numbers.
The symbol ``:='' means that
the expression on the right is the definition of that on the left.
For two positive integers $k$ and $l$, $F^{(k,l)}$ denotes the set
of all $k\times l$ matrices with entries in a commutative ring
$F$. For a square matrix $A\in F^{(k,k)}$ of degree $k$,
$Tr(A)$ denotes the trace of $A$. For any $M\in F^{(k,l)},\
^t\!M$ denotes the transpose of $M$. For a positive integer $n$, $I_n$
denotes the identity matrix of degree $n$.
For $A\in F^{(k,l)}$ and $B\in
F^{(k,k)}$, we set $B[A]=\,^tABA$ (Siegel's notation). For a complex matrix $A$,
${\overline A}$ denotes the complex {\it conjugate} of $A$.
${\rm diag}(a_1,\cdots,a_n)$ denotes the $n\times n$ diagonal matrix with diagonal entries
$a_1,\cdots,a_n$. For a smooth manifold, we denote by $C_c (X)$ (resp. $C_c^{\infty}(X)$ the algebra of all continuous (resp. infinitely differentiable) functions on $X$ with compact support.
\begin{equation*}
J_g=\begin{pmatrix} 0&I_g\\
                   -I_g&0\end{pmatrix}
\end{equation*}
denotes the symplectic matrix of degree $2g$.
\begin{equation*}
{\mathbb H}_g=\,\{\,\Omega\in \BC^{(g,g)}\,|\ \Omega=\,^t\Omega,\ \ \ \text{Im}\,\Omega>0\,\}
\end{equation*}
denotes the Siegel upper half plane of degree $g$.
\begin{equation*}
  Sp(g,\BR)=\{ M\in \BR^{(2g,2g)}\,|\ {}^tM J_gM=J_g\,\}
\end{equation*}
denotes the symplectic group of degree $g$ and
\begin{equation*}
  \G_g^{\flat}=\{ \gamma\in \BZ^{(2g,2g)}\,|\ {}^t\gamma J_g\gamma=J_g\,\}\subset Sp(g,\BR)
\end{equation*}
denotes the Siegel modular group of degree $g$.

\end{section}

\newcommand\POB{ {{\partial}\over {\partial{\overline \Omega}}} }
\newcommand\PZB{ {{\partial}\over {\partial{\overline Z}}} }
\newcommand\PX{ {{\partial}\over{\partial X}} }
\newcommand\PY{ {{\partial}\over {\partial Y}} }
\newcommand\PU{ {{\partial}\over{\partial U}} }
\newcommand\PV{ {{\partial}\over{\partial V}} }
\newcommand\PO{ {{\partial}\over{\partial \Omega}} }
\newcommand\PZ{ {{\partial}\over{\partial Z}} }
\newcommand\PW{ {{\partial}\over{\partial W}} }
\newcommand\PWB{ {{\partial}\over {\partial{\overline W}}} }
\newcommand\OVW{\overline W}
\newcommand\Rg{{\mathfrak R}_n}

\vskip 10mm
\begin{section}{{\bf Review on the geometry of $GL(n,\BR)/O(n,\BR)$ }}
\setcounter{equation}{0}

\vskip 0.3cm
For $Y=(y_{ij})\in \CP,$ we put
\begin{equation*}
dY=(dy_{ij})\qquad\text{and}\qquad \PY\,=\,\left(\, {
{1+\delta_{ij}}\over 2}\, { {\partial}\over {\partial y_{ij} } }
\,\right).
\end{equation*}

\vskip 0.2cm For a fixed element $A\in GL(n,\BR)$, we put
$$Y_*=A\cdot Y=AY\,^t\!A,\quad Y\in \CP.$$
Then
\begin{equation}
dY_*=A\,dY\,^t\!A \quad \textrm{and}\quad {{\partial}\over {\partial
Y_*}}=\,^t\!A^{-1} \Yd\, A^{-1}.
\end{equation}

\vskip 5mm
We can see easily that
\begin{equation*}  ds^2=\,Tr( (Y^{-1}dY)^2)  \end{equation*}
is a $GL(n,\BR)$-invariant Riemannian metric on $\CP$ and its
Laplacian is given by
\begin{equation*}
\Delta=   Tr\left( \left( Y\PY\right)^2\right),
\end{equation*}
\noindent where $ Tr(M)$ denotes the trace of a square
matrix $M$. We also can see that
\begin{equation}
d\mu_n(Y)=(\det Y)^{-{ {n+1}\over2 } }\prod_{i\leq j}dy_{ij}
\end{equation}
is a $GL(n,\BR)$-invariant volume element on $\CP$.

\begin{theorem}
A geodesic $\alpha (t)$ joining $I_n$ and $Y\in \CP$ has the form
\begin{equation*}
\alpha (t)=\exp (t A[V]),\qquad t\in [0,1],
\end{equation*}
where
\begin{equation*}
Y=(\exp A)[V]=\exp (A[V])=\exp (\,^tVAV)
\end{equation*}
is the spectral decomposition of $Y$, where $V\in O(n,\BR),\ A={\rm diag} (a_1,\cdots,a_n)$ with all $a_j\in \BR.$
The distance of $\alpha (t) \ (0\leq t\leq 1)$ between $I_n$ and $Y$ is
\begin{equation*}
 \left( \sum_{j=1}^{n} a_j^2 \right)^{\frac{1}{2}}.
\end{equation*}
\end{theorem}
\begin{proof} The proof can be found in \cite[pp.\,16-17]{T}.\end{proof}

\vskip 2mm
We consider the following differential operators
\begin{equation}
D_k= Tr\left( \left( Y\Yd \right)^k\right),\quad
k=1,2,\cdots,n,
\end{equation}
By Formula (2.1), we get
\begin{equation*}
\left( Y_* {{\partial}\over {\partial Y_*}}\right)^i=\,A\,\left(
Y\Yd\right)^i A^{-1}
\end{equation*}

\noindent for any $A\in GL(n,\BR)$. So each $D_i\ (1\leq i \leq n)$ is invariant
under the action (1.1) of $GL(n,\BR)$.

\vskip 0.2cm Selberg \cite{S1} proved the following.

\begin{theorem}
The algebra ${\mathbb D}(\CP)$ of all $GL(n,\BR)$-invariant differential operators on
$\CP$ is generated by $D_1,D_2,\cdots,D_n.$ Furthermore $D_1,D_2,\cdots,D_n$ are algebraically independent and ${\mathbb D}(\CP)$ is isomorphic to the commutative ring $\BC[x_1,x_2,\cdots,x_n]$ with $n$ indeterminates $x_1,x_2,\cdots,x_n.$
\end{theorem}

\begin{proof} The proof can be found in \cite[pp.\,64-66]{M2}.\end{proof}

\vskip 3mm
For $s=(s_1,\cdots,s_n)\in \BC^n$, Atle Selberg \cite[pp.\,57-58]{S1} introduced the power function
$p_s:\CP\lrt \BC$ defined by
\begin{equation}
  p_s (Y):=\prod_{j=1}^{n} (\det Y_j)^{s_j},\quad Y\in \CP,
\end{equation}
where $Y_j\in \mathscr P_j\ (1\leq j\leq n)$ is the $j\times j$ upper left corner of $Y$.
Let
\begin{equation}
T_n:=\left\{\, t=
\begin{pmatrix}
           t_{11} & t_{12} & \cdots & t_{1n} \\
           0 & t_{22} & \cdots & t_{2n}\\
           0 & 0 & \ddots &  \vdots\\
           0 & 0 &   0  & t_{nn}
         \end{pmatrix}\in GL(n,\BR)\,\Big|\ t_{jj}>0,\ 1\leq j\leq n \ \right\}
\end{equation}
be the subgroup of $GL(n,\BR)$ consisting of upper triangular matrices. For $r=(r_1,\cdots,r_n)\in \BC^n$,
we define the group homomorphism $\tau_r : T_n\lrt \BC^{\times}$ by
\begin{equation}
  \tau_r (t):=\prod_{j=1}^{n} t_{jj}^{r_j},\quad t=(t_{ij})\in T_n.
\end{equation}
For $z=(z_1,\cdots,z_n)\in \BC^n$, we define the function $\phi_z: T_n\lrt \BC^{\times}$ by
\begin{equation}
  \phi_z (t):= \prod_{j=1}^{n} t_{jj}^{2 z_j+j-\frac{n+1}{2}},\quad t=(t_{ij})\in T_n.
\end{equation}
We note that $\phi_z (t)=p_s (I_n[t])$ for some $s\in\BC^n$.

\begin{proposition}
(1) For $s=(s_1,\cdots,s_n)\in \BC^n$, we put $r_j=2(s_j+\cdots + s_n),\ j=1,\cdots, n.$
Then we have
\begin{equation*}
  p_s (I_n [t])=\tau_r (t),\quad t\in T_n.
\end{equation*}
\vskip2mm
\noindent
(2) $p_s (Y[t])=p_s(I_n[t])\,p_s(Y)$ for any $Y\in \CP$ and $t\in T_n$.
\vskip 2mm
\noindent
(3) For any $D\in \BD (\CP)$, we have $Dp_s=Dp_s (I_n)\,p_s,$ i.e., $p_s$ is a common eigenfunction
of $\BD (\CP)$.
\end{proposition}
\begin{proof} The proof can be found in \cite[pp.\,39-40]{T}.\end{proof}

\vskip 3mm
Hans Maass \cite{M2} proved the following theorem.
\begin{theorem}
(1) Let $D_1,\cdots,D_n\in  \BD (\CP)$ be algebraically independent invariant differential operators given by Formula (2.3). Then
\begin{equation*}
  D_j \phi_z = \lambda_j (z) \phi_z, \quad 1\leq j\leq n.
\end{equation*}
where $\lambda_j (z)$ is a symmetric polynomial in $z_1,\cdots, z_n$ of degree $j$ and having the following form:
\begin{equation*}
\lambda_j (z_1,\cdots, z_n)=z_1^j+\cdots+z_n^j + {\rm terms\ of\ lower\ degree}.
\end{equation*}
\noindent
(2) The effect of $D\in \BD (\CP)$ on power functions $p_s (Y)$ determines $D$ uniquely.
\end{theorem}
\begin{proof} The proof can be found in \cite[pp.\,70-76]{M2} or \cite[pp.\,44-48]{T}.\end{proof}

\vskip 3mm
A function $h:\CP\lrt \BC$ is said to be $\textsf{spherical}$ if $h$ satisfies the following properties
(2.8)-(2.10):
\begin{equation}
  h(Y[k])=h(\,^tk Yk)=h(Y)\quad {\rm for\ all}\ Y\in \CP\ {\rm and}\ k\in O(n,\BR).
\end{equation}
\begin{equation}
h\ {\rm is\ a\ common\ eigenfunction\ of} \ \BD (\CP).
\end{equation}
\begin{equation}
h(I_n)=1.
\end{equation}

\vskip 2mm
For the present, we put $G=GL(n,\BR)$ and $K=O(n,\BR)$. For $s=(s_1,\cdots,s_n)\in \BC^n$, we define the function
\begin{equation}
h_s (Y):=\int_{K} p_s(Y[k])\,dk, \quad Y\in \CP,
\end{equation}
where $dk$ is a normalized measure on $K$ so that $\int_K dk =1.$ It is easily seen that $h_s (Y)$ is a spherical function on $\CP.$ Selberg \cite[pp.\,53-57]{S1} proved that these $h_s (Y)$ are the only spherical functions on $\CP.$ If $f\in C_c (\CP)$, the $\textsf{Helgason-Fourier\ transform}$ of $f$ is defined to be the function ${\mathscr H}f:\BC^n\times K\lrt \BC$\,:
\begin{equation}
{\mathscr H}f (s,k):=\int_{\CP} f(Y)\,\overline{p_s(Y[k])}\,d\mu_n (Y), \quad (s,k)\in \BC^n\times K,
\end{equation}
where $p_s$ is the Selberg power function (see Formula (2.4)) and $d\mu_n (Y)$ is a $GL(n,\BR)$-invariant volume element on $\CP$ (see Formula (2.2)).

\begin{proposition}
(1) The spherical function $h$ on $\CP$ corresponding to the eigenvalues $\lambda_1,\cdots,\lambda_n\in \BC$ with
\begin{equation*}
D_i h= \lambda_i h,\quad 1\leq i\leq n
\end{equation*}
is unique. Here $D_1,\cdots,D_n$ are invariant differential operators on $\CP$ defined by Formula (2.3).
\vskip 2mm \noindent
(2) Let $f\in C^c (\CP)$ be a common eigenfunction of $\BD (\CP)$, i.e., $Df=\lambda_D f \,(\lambda_D\in \BC)$ for all $D\in \BD (\CP).$ Define $s\in \BC^n$ by
\begin{equation*}
Dp_s=\lambda_D p_s, \quad D\in \BD (\CP).
\end{equation*}
If $g\in C^{\infty}_c (K\backslash G/K)$ is a $K$-bi-invariant function on $G$ satisfying the condition $g(x)=g(x^{-1})$ for all $x\in G$, then
\begin{equation*}
f \star g = {\hat g}({\bar s}) f,
\end{equation*}
where $\star$ denotes the convolution operator and
\begin{equation*}
{\hat g}({\bar s}):=\int_{\CP} g(Y)\,\overline{p_{\bar s}(Y[k])}\,d\mu_n (Y).
\end{equation*}
Conversely, suppose that $f\in C^{\infty}(\CP)$ is a $K$-invariant eigenfunction of all convolution operators with $g\in C^{\infty}_c (\CP).$ Then $f$ is a common eigenfunction of $\BD (\CP)$.
\end{proposition}
\begin{proof} The proof can be found in \cite[pp.\,53-56]{S1} or \cite[pp.\,67-69]{T}.\end{proof}

\vskip 3mm
The fundamental domain ${\mathfrak R}_n$ for $GL(n,\BZ)\ba \CP$ which was
found by H. Minkowski \cite{Mi} is defined as a subset of $\CP$
consisting of $Y=(y_{ij})\in \CP$ satisfying the following
conditions (M.1)--(M.2)\ (cf.\,\cite[p.\,123]{M2}):

\vskip 0.1cm
(M.1)\ \ \ $aY\,^ta\geq
y_{kk}$\ \ for every $a=(a_i)\in\BZ^n$ in which $a_k,\cdots,a_n$
are relatively prime for $k=1,2,\cdots,n$.

\vskip 0.1cm
(M.2)\ \ \
\ $y_{k,k+1}\geq 0$ \ for $k=1,\cdots,n-1.$

\vskip 0.1cm
We say
that a point of $\Rg$ is {\it Minkowski reduced} or simply {\it
M}-{\it reduced}. $\Rg$ has the following properties (R1)-(R4):

\vskip 0.1cm
(R1) \ For any $Y\in\CP,$ there exist a matrix $A\in
GL(n,\BZ)$ and $R\in\Rg$ such that $Y=R[A]$\,(cf.\,\cite[p.\,139]{M2}). That is,
\begin{equation*}
  GL(n,\BZ)\circ \Rg=\CP.
\end{equation*}
\indent (R2)\ \ $\Rg$ is a convex cone through the origin bounded
by a finite number of hyperplanes. $\Rg$ is closed in $\CP$
(cf.\,\cite[p.\,139]{M2}).

\vskip 0.1cm
(R3) If $Y$ and $Y[A]$ lie in $\Rg$ for $A\in
GL(g,\BZ)$ with $A\neq \pm I_n,$ then $Y$ lies on the boundary
$\partial \Rg$ of $\Rg$. Moreover $\Rg\cap (\Rg [A])\neq
\emptyset$ for only finitely many $A\in GL(n,\BZ)$
(cf.\,\cite[p.\,139]{M2}).

\vskip 0.1cm (R4) If $Y=(y_{ij})$ is
an element of $\Rg$, then
\begin{equation*}
y_{11}\leq y_{22}\leq \cdots \leq y_{nn}\quad \text{and}\quad
|y_{ij}|<{\frac 12}y_{ii}\quad \text{for}\ 1\leq i< j\leq n.
\end{equation*}
\indent We refer to \cite[pp.\,123-124]{M2}.

\vskip 0.3cm
$\CP$ parameterizes principally polarized real tori of dimension $n$ (see Section 3). The arithmetic quotient $GL(n,\BZ)\backslash \CP$ is the moduli space of isomorphism classes of principally polarized
real tori of dimension $n$.
According to (R2) we see that $\Rg$ is a semi-algebraic set with real analytic structure.

\end{section}

\newcommand\FA{\mathfrak A}
\newcommand\FL{\mathfrak L}
\newcommand\FT{\mathfrak T}
\newcommand\BH{\mathbb H}
\newcommand\Om{\Omega}
\newcommand\La{\Lambda}

\vskip 10mm
\begin{section}{{\bf Polarized real tori }}
\setcounter{equation}{0}

\vskip 3mm
In this section, we recall the concept of polarized real tori (cf.\,\cite{Y3}).

We review basic notions and some results on real principally polarized abelian varieties
(cf.\,\cite{GoT, S-S, Si1, Si2, Si3}).

\begin{definition}
A pair $(\FA,S)$ is said to be a \textsf{real abelian variety} if $\FA$ is a complex abelian variety and $S$ is an
anti-holomorphic involution of $\FA$ leaving the origin of $\FA$ fixed. The set of all fixed points of $S$ is called the
{\it real point} of $(\FA,S)$ and denoted by $(\FA,S)(\BR)$ or simply $\FA(\BR)$. We call $S$ a
\textsf{real structure} on $\FA$.
\end{definition}

\begin{definition}
(1) A \textsf{polarization} on a complex abelian variety $\FA$ is defined to be the Chern class
$c_1(D)\in H^2(\FA,\BZ)$ of an ample divisor $D$ on $\FA$. We can identify $H^2(\FA,\BZ)$ with
$\bigwedge^2 H^1(\FA,\BZ)$. We write $\FA=V/L$, where $V$ is a finite dimensional complex vector space and $L$ is
a lattice in $V$. So a polarization on $\FA$ can be defined as an alternating form $E$ on $L\cong H_1(\FA,\BZ)$
satisfying the following conditions (E1) and (E2)\,:
\vskip 0.1cm \noindent
(E1) The Hermitian form $H:V\times V\lrt \BC$ defined by
\begin{equation}
H(u,v)=\,E(i\,u,v)\,+\,i\, E(u,v),\qquad u,v\in V
\end{equation}
\noindent is positive definite. Here $E$ can be extended $\BR$-linearly to an alternating form on $V$.
\vskip 0.1cm \noindent
(E2) $E(L\times L)\subset \BZ$, i.e., $E$ is integral valued on $L\times L.$
\vskip 0.2cm\noindent
(2) Let $(\FA,S)$ be a real abelian variety with a polarization $E$ of dimension $g$.
A polarization $E$ is said to be
\textsf{real} or $S$-\textsf{real} if
\begin{equation}
E(S_*(a),S_*(b))=\,-E(a,b),\qquad a,b\in H_1(\FA,\BZ).
\end{equation}
Here $S_*:H_1(\FA,\BZ) \lrt H_1(\FA,\BZ)$ is the map induced by a real structure $S$.
If a polarization $E$ is real, the triple $(\FA,E,S)$ is called a \textsf{real polarized abelian variety}.
A polarization $E$ on $\FA$ is said to be \textsf{principal} if for a suitable basis (i.e., a symplectic
basis) of $H_1(\FA,\BZ)\cong L$, it is represented by the symplectic matrix $J_g$ (cf.\,see Notations in the
introduction).
A real abelian variety $(\FA,S)$ with a principal polarization $E$ is called a
\textsf{real principally polarized abelian variety}.

\vskip 0.2cm\noindent
(3) Let $(\FA,E)$ be a principally polarized abelian variety of dimension $g$ and let
$\{ \alpha_i\,|\ 1\leq i\leq 2g\,\}$ be a symplectic basis of $H_1(\FA,\BZ)$. It is known that there is
a basis $\{ \omega_1,\cdots,\omega_g \}$ of the vector space $H^0(\FA,\Omega^1)$ of holomorphic 1-forms on $\FA$
such that
\begin{equation*}
\left(\, \int_{\alpha_j}\omega_i\,\right) =\,(\Omega,I_g)\qquad \textrm{for some}\ \Omega\in {\mathbb H}_g.
\end{equation*}
The $g\times 2g$ matrix $(\Omega,I_g)$ or simply $\Omega$ is called a \textsf{period matrix} for $(\FA,E).$
\end{definition}

\vskip 0.3cm The definition of a {\it real polarized abelian variety} is motivated by the following theorem.
\begin{theorem}
Let $(\FA,S)$ be a real abelian variety and let $E$ be a polarization on $\FA$. Then there exists an ample
$S$-invariant (or $S$-real) divisor with Chern class $E$ if and only if $E$ satisfies the condition (3.2).
\end{theorem}
\noindent
{\it Proof.} The proof can be found in \cite[Theorem 3.4, pp.\,81-84]{Si2}. \hfill $\square$

\vskip 0.3cm
Now we consider a principally polarized abelian variety of dimension $g$ with a level structure.
Let $N$ be a positive integer. Let $(\FA=\,\BC^g/L,E)$ be a principally polarized abelian variety of dimension $g$.
From now on we write $\FA=\,\BC^g/L$, where $L$ is a lattice in $\BC^g$. A \textsf{level} $N$
\textsf{structure} on $\FA$ is a choice of a basis $\{ U_i,V_j\}\,(1\leq i,j\leq g)$ for a $N$-torsion points of $\FA$
which is symplectic, in the sense that there exists a symplectic basis $\{ u_i,v_j\}$ of $L$ such that
\begin{equation*}
U_i\equiv {{u_i}\over N}\ (\textrm{mod}\,L)\qquad \textrm{and}\qquad
V_j\equiv {{v_j}\over N}\ (\textrm{mod}\,L),\qquad 1\leq i,j\leq g.
\end{equation*}
For a given level $N$ structure, such a choice of a symplectic basis $\{ u_i,v_j\}$ of $L$ determines a
mapping
$$F:\BR^g\oplus \BR^g\lrt \BC^g$$
such that $F(\BZ^g\oplus \BZ^g)=\,L$ by $F(e_i)=u_i$ and $F(f_j)=v_j$,
where $\{e_i,f_j\}\,(1\leq i,j\leq g)$
is the standard basis of $\BR^g\oplus\BR^g.$ The choice $\{ u_i,v_j\}$ (or equivalently, the mapping $F$) will be referred to as a {\it lift} of the level $N$ structure.
Such a mapping $F$ is well defined modulo the principal
congruence subgroup $\G_g(N)$, that is, if $F'$ is another lift of the level structure, then
$F'\circ F^{-1}\in \G_g(N).$ A level $N$ structure $\{ U_i,V_j\}$ is said to be $ \textsf{compatible}$ with
a real structure $S$ on $(\FA,E)$ if, for some (and hence for any) lift $\{ u_i,v_j\}$ of the level structure,
\begin{equation*}
S\left( {{u_i}\over N}\right)\equiv -{{u_i}\over N}\ (\textrm{mod}\,L)\qquad \textrm{and}\qquad
S \left( {{v_j}\over N}\right) \equiv {{v_j}\over N}\ (\textrm{mod}\,L),\qquad 1\leq i,j\leq g.
\end{equation*}

\begin{definition}
A real principally polarized abelian variety of dimension $g$ with a level $N$ structure is a quadruple
${\mathcal A}=\,(\FA,E,S, \{ U_i,V_j\})$ with $\FA=\BC^g/L$, where $(\FA,E,S)$ is a real principally polarized abelian variety
and $\{ U_i,V_j\}$ is a level $N$ structure compatible with a real structure $S$. An isomorphism
$${\mathcal A}=\,(\FA,E,S, \{ U_i,V_j\})\,\cong\,(\FA',E',S', \{ U_i',V_j'\})={\mathcal A}'$$
is a complex linear mapping $\phi:\BC^g\lrt \BC^g$ such that
\begin{equation}
\phi(L)=\,L',
\end{equation}
\begin{equation}
\phi_*(E)=\,E',
\end{equation}
\begin{equation}
\phi_*(S)=\,S',\ that\ is,\ \ \phi\circ S\circ \phi^{-1}=S',
\end{equation}
\begin{equation}
\phi \left( {{u_i}\over N}\right)\equiv {{u_i'}\over N}\ (\textrm{mod}\,L')\qquad \textrm{and}\qquad
\phi \left( {{v_j}\over N}\right) \equiv {{v_j'}\over N}\ (\textrm{mod}\,L'),\qquad 1\leq i,j\leq g.
\end{equation}
\noindent for some lift $\{ u_i,v_j\}$ and $\{ u_i',v_j'\}$ of the level structures.
\end{definition}

Now we show that a given positive integer $N$ and a given $\Omega\in \BH_g$ determine naturally a
principally polarized abelian variety $(\FA_\Om,E_\Om)$ of dimension $g$ with a level $N$ structure.
Let $E_0$ be the standard alternating form on $\BR^g\oplus\BR^g$ with the symplectic matrix $J_g$
with respect to the standard basis of $\BR^g\oplus\BR^g$. Let $F_\Om:\BR^g\oplus\BR^g\lrt \BC^g$
be the real linear mapping with matrix $(\Om,I_g)$, that is,
\begin{equation}
F_\Om \begin{pmatrix} x \\ y \end{pmatrix}:=\Om\, x+y,\qquad x,y\in\BR^g.
\end{equation}
We define $E_\Om:=\,(F_\Om)_*(E_0)$ and $L_\Om:=\,F_\Om (\BZ^g\oplus\BZ^g).$ Then
$(\FA_\Om=\,\BC^g/L_\Om,E_\Om)$ is a principally polarized abelian variety. The Hermitian form
$H_\Om$ on $\BC^g$ corresponding to $E_\Om$ is given by
\begin{equation}
H_\Om (u,v)=\, {}^tu\,( \textrm{Im}\,\Om)^{-1}\,{\overline v},\quad E_\Om= \textrm{Im}\,H_\Om,
\qquad u,v\in\BC^g.
\end{equation}
\noindent
If $z_1,\cdots,z_g$ are the standard coordinates on $\BC^g$, then the holomorphic 1-forms $dz_1,\cdots,
dz_g$ have the period matrix $(\Om,I_g).$ If $\{e_i,f_j\}$ is the standard basis of $\BR^g\oplus\BR^g$,
then \\ $\left\{ F_\Om(e_i/N),\,F_\Om(f_j/N)\right\}$
(mod\ $L_\Om$) is a level $N$ structure on
$(\FA_\Om,E_\Om)$, which we refer to as the {\it standard} $N$ {\it structure}. Assume that $\Om_1$ and
$\Om_2$ are two elements of $\BH_g$ such that
$$\psi: (\FA_{\Om_1}=\BC^g/L_{\Om_1},E_{\Om_1})\lrt (\FA_{\Om_2}=\BC^g/L_{\Om_2},E_{\Om_2})$$
is an isomorphism of the corresponding principally polarized abelian varieties, i.e.,
$\psi(L_{\Om_1})=L_{\Om_2}$ and $\psi_*(E_{\Om_1})=\,E_{\Om_2}.$ We set
$$ h=\,{}^t\big( F_{\Om_2}^{-1}\circ \psi\circ F_{\Om_1}\big)=\,
\begin{pmatrix} A & B\\
C & D \end{pmatrix}.$$
Then we see that $h\in \G_g$. And we have
\begin{equation}
\Om_1=\,h\cdot\Om_2 =\,(A\Om_2+B)(C\Om_2+D)^{-1}
\end{equation}
and
\begin{equation}
\psi(Z)=\,{}^t(C\Om_2+D)Z,\qquad Z\in\BC^g.
\end{equation}

\vskip 0.2cm
Let
\begin{equation*}
  I_*:=  \begin{pmatrix}
           -I_n & 0 \\
          \,0 & I_n
         \end{pmatrix}.
\end{equation*}
We define the involution $\tau:Sp(g,\BR)\lrt Sp(g,\BR)$ by
\begin{equation*}
  \tau(x):=I_* x I_*,\quad x\in Sp(g,\BR).
\end{equation*}
Precisely $\tau$ is given by
\begin{equation*}
  \tau \begin{pmatrix}
         A & B \\
         C & D
       \end{pmatrix}
       =\begin{pmatrix}
         \,A & -B \\
         -C & \,D
       \end{pmatrix},\quad
       \begin{pmatrix}
         A & B \\
         C & D
       \end{pmatrix}\in Sp(g,\BR).
\end{equation*}
We note that $\tau:Sp(g,\BR)\lrt Sp(g,\BR)$ passes to an involution
(which we denote by the same letter) $\tau:\BH_g\lrt\BH_g$ such that
\begin{equation*}
  \tau(x\cdot \Omega)=\tau(x)\tau(\Omega)\quad {\rm for\ all}\ x\in Sp(g,\BR),\ \Om\in \BH_g.
\end{equation*}
In fact, we can see easily that the involution $\tau:\BH_g\lrt\BH_g$ is the antiholomorphic
involution given by
\begin{equation*}
  \tau(\Omega)=-\overline{\Omega},\quad \Om\in \BH_g.
\end{equation*}
Its fixed point set is the orbit
\begin{equation*}
  i\mathscr{P}_g=GL(g,\BR)\cdot (iI_g)\subset \BC^{(g,g)}.
\end{equation*}
We refer to \cite[pp.\,274-275,\ 278]{Y3}.
\vskip 2mm
Let $\Om\in\BH_g$ such that $\gamma\cdot \Om=\,\tau(\Om)=\,-{\overline \Om}$ for some
$\gamma=\begin{pmatrix} A & B\\
C & D \end{pmatrix}\in \G_g^{\flat}.$ We define the mapping $S_{\gamma,\Om}:\BC^g\lrt\BC^g$ by
\begin{equation}
S_{\gamma,\Om}(Z):=\,{}^t(C\Om+D)\,{\overline Z},\qquad Z\in\BC^g.
\end{equation}
Then we can show that $S_{\gamma,\Om}$ is a real structure on $(\FA_\Om,E_\Om)$ which is compatible
with the polarization $E_\Om$ (that is, $E_\Om(S_{\gamma,\Om}(u),S_{\gamma,\Om}(v))=-E_\Om(u,v)$
for all $u,v\in \BC^g$). Indeed according to Comessatti's Theorem (see Theorem 3.1), $S_{\gamma,\Om}(Z)=\,
{\overline Z},$ i.e., $S_{\gamma,\Om}$ is a complex conjugation. Therefore we have
$$E_\Om(S_{\gamma,\Om}(u),S_{\gamma,\Om}(v))=\,E_\Om({\overline u},{\overline v})=-E_\Om(u,v)$$
for all $u,v\in \BC^g$. From now on we write simply $\s_\Om=\,S_{\gamma,\Om}.$

\begin{theorem} Let $(\FA,E,S)$ be a real principally polarized abelian variety of
dimension $g$. Then there exists $\Om=\,X +\, i\,Y \in \BH_g$ such that $2X\in \BZ^{(g,g)}$ and there
exists an isomorphism of real principally polarized abelian varieties
$$(\FA,E,S)\,\cong\,(\FA_\Om,E_\Om,\sigma_\Om),$$
where $\sigma_\Om$ is a real structure on $\FA_\Om$ induced by a complex conjugation $\sigma:\BC^g\lrt\BC^g$.
\end{theorem}

The above theorem is essentially due to Comessatti \cite{C}. We refer to \cite{Si1, Si2} for the proof of
Theorem 3.2.

\vskip 0.5cm Theorem 3.2 leads us to define the subset ${\mathscr H}_g$ of $\BH_g$ by
\begin{equation}
{\mathscr H}_g:=\,\left\{\,\Om\in\BH_g\,\,|\ \,2\,\textrm{Re}\,\Om \in \BZ^{(g,g)}\ \right\}.
\end{equation}
We note that ${\mathscr H}_g$ is a countable union of analytic subsets (see [26] or [28], p.\,280).
Precisely
\begin{equation*}
  {\mathscr H}_g=\left\{ \Omega=x+i\,y={}^t\Omega\,|\ x,y\in \BR^{(g,g)},\
   y>0,\ 2x_{ij}\in\BZ,\ y_{ij}\in \BR,
  \ 1\leq i\leq j\leq g\right\},
\end{equation*}
where $x=(x_{ij})={}^tx \in \BR^{(g,g)}$ and $y=(y_{ij})={}^ty\in \BR^{(g,g)}$ with $y>0$. Or
\begin{equation*}
  {\mathscr H}_g=\bigcup_{2x\in S_g(\BZ)} \left( x+i\,{\mathscr P}_g\right),
\end{equation*}
where
$$S_g(\BZ):=\left\{ x\in \BR^{(g,g)}\,|\ x \ is\ integeral,\ x={}^tx\,\right\}$$
and
$${\mathscr P}_g:=\left\{ y\in \BR^{(g,g)}\,|\ x={}^tx >0,\ positive\ definite\,\right\}.$$
${\mathscr H}_g$ may be considered as a countable union of semialgebraic subsets.
\vskip 3mm
Assume $\Om=\,X+\,i\,Y\in {\mathscr H}_g.$ Then according to Theorem 3.2, $(\FA_\Om,E_\Om,\sigma_\Om)$ is a real
principally polarized abelian variety of dimension $g$. The matrix $M_\s$ for the action of a complex conjugation
$\s$ on the lattice $L_\Om=\Om\BZ^g+\BZ^g$ with respect to the basis given by the columns of $(\Om,I_g)$ is given by
\begin{equation}
M_\s=\,\begin{pmatrix} -I_g & 0 \\ 2\,X & I_g \end{pmatrix}.
\end{equation}
Since
$${}^tM_\s\, J_g\, M_\s=\,\begin{pmatrix} -I_g & 2\,X \\ 0 & I_g \end{pmatrix}\,J_g\,\begin{pmatrix} -I_g & 0 \\ 2\,X & I_g \end{pmatrix}
=-J_g,$$
the canonical polarization $J_g$ is $\s$-real.

\begin{theorem}
Let $\Om$ and $\Om_*$ be two elements in ${\mathscr H}_g$. Then $\Om$ and $\Om_*$ represent (real) isomorphic triples
$(\FA,E,\s)$ and $(\FA_*,E_*,\s_*)$ if and only if there exists an element $A\in GL(g,\BZ)$ such that
\begin{equation}
2\, \textrm{Re}\,\,\Om_*=\,2\, A\,(\textrm{Re}\,\,\Om)\,\,{}^tA\ \ (\textrm{mod}\,\,2)
\end{equation}
and
\begin{equation}
\textrm{Im}\,\,\Om_*=\,A\,(\textrm{Im}\,\,\Om)\,\,{}^tA.
\end{equation}
\end{theorem}
\noindent
{\it Proof.} Suppose $(\FA,E,\s)$ and $(\FA_*,E_*,\s_*)$ are real isomorphic. Then we can find an element
$\gamma=\begin{pmatrix} A & B\\
C & D \end{pmatrix}\in \G_g$ such that
$$\Om_*=\,(A\Om+B)(C\Om+D)^{-1}.$$
The map
$$\varphi: \BC^g/L_{\Om_*}=\,\FA_{\Om_*}\lrt \FA_\Om=\,\BC^g/L_\Om$$
induced by the map
$${\widetilde\varphi}:\BC^g\lrt\BC^g,\qquad Z\longmapsto \,{}^t(C\Om+D)\,Z$$
is a real isomorphism. Since ${\widetilde\varphi}\circ \s_* =\,\s\circ {\widetilde\varphi},$ i.e.,
${\widetilde\varphi}$ commutes with complex conjugation on $\BC^g$, we have $C=0.$ Therefore
$$\Om_*=\,(A\Om+B)\,{}^t\!A=\,(AX\,{}^t\!A+B\,{}^t\!A)+\,i\,AY\,{}^t\!A,$$
where $\Om=\,X+\,i\,Y.$ Hence we obtain the desired results (3.14) and (3.15).
\vskip 0.1cm Conversely we assume that there exists $A\in GL(g,\BZ)$ satisfying the conditions
(3.14) and (3.15). Then
$$\Om_*=\,\gamma\cdot \Om=\,(A\Om+B)\,{}^t\!A$$
for some $\gamma=\begin{pmatrix} A & B\\
0 & {}^tA^{-1} \end{pmatrix}\in \G_g$ with $B\in\BZ^{(g,g)}$ with $B\,{}^tA=\,A\,{}^tB.$ The map
$\psi:\,\FA_\Om\lrt \FA_{\Om_*}$ induced by the map
$${\widetilde\psi}:\BC^g\lrt \BC^g,\qquad Z\longmapsto A^{-1}Z$$
is a complex isomorphism commuting complex conjugation $\s$. Therefore $\psi$ is a real isomorphism
of $(\FA,E,\s)$ onto $(\FA_*,E_*,\s_*)$. \hfill $\square$

\vskip 0.3cm
According to Theorem 3.3, we are led to define the subgroup $\G_g^{\star}$ of $\G_g$ by
\begin{equation}
\G_g^{\star}:=\,\left\{\, \begin{pmatrix} A & B\\
0 & {}^tA^{-1} \end{pmatrix}\in \G_g\ \big|\ \ B\in \BZ^{(g,g)},\quad A\,{}^tB=\,B\,{}^t\!A\ \right\}.
\end{equation}
It is easily seen that $\G_g^{\star}$ acts on ${\mathscr H}_g$ properly discontinuously by
\begin{equation}
\gamma\cdot \Om=\,A\Om\,{}^t\!A\,+\,B\,{}^t\!A,
\end{equation}
where $\gamma=\begin{pmatrix} A & B\\
0 & {}^tA^{-1} \end{pmatrix}\in \G_g^{\star}$ and $\Om\in {\mathscr H}_g.$

\vskip 3mm
Now we define the notion of polarized real tori.
\begin{definition}
A real torus $T=\BR^n/\Lambda$ with a lattice $\La$ in $\BR^n$ is said to be $\textsf{polarized}$
if the the associated complex torus $\FA=\BC^n/L$ is a polarized real abelian variety, where
$L=\,\BZ^n+\,i\,\La$ is a lattice in $\BC^n.$ Moreover if $\FA$ is a principally polarized real abelian
variety, $T$ is said to be \textsf{principally polarized}.
Let $\Phi:T\lrt\FA$ be the smooth embedding of $T$ into $\FA$ defined by
\begin{equation}
\Phi(v+\La):=\,i\,v\,+\,L, \qquad v\in\BR^n.
\end{equation}
Let $\FL$ be a polarization of $\FA$, that is, an ample line bundle over $\FA$. The pullback $\Phi^*\FL$ is
called a \textsf{polarization} of $T$. We say that a pair $(T,\Phi^*\FL)$ is a \textsf{polarized real torus}.
\end{definition}

\noindent{\bf Example 3.1.} Let $Y\in \CP$ be a $n\times n$ positive definite symmetric real matrix. Then
$\La_Y=\,Y\BZ^n$ is a lattice in $\BR^n$. Then the $n$-dimensional torus $T_Y=\BR^n/\La_Y$ is a
principally polarized real torus. Indeed,
\begin{equation*}
\FA_Y\,=\,\BC^n/L_Y, \qquad  L_Y\,=\BZ^n+\,i\,\La_Y
\end{equation*}
is a princially polarized real abelian variety.
Its corresponding hermitian form $H_Y$ is given by
\begin{equation*}
H_Y(x,y)\,=\,E_Y(i\,x,y)\,+\,i\,E_Y(x,y)\,=\,{}^tx\,Y^{-1}\,{\overline y},\qquad x,y\in\BC^n,
\end{equation*}
where
$E_Y$ denotes the imaginary part of $H_Y.$ It is easily checked that $H_Y$ is positive definite and
$E_Y(L_Y\times L_Y)\subset \BZ$\,(cf.\,\cite[pp.\,29--30]{Mum}).
The real structure $\s_Y$ on $\FA_Y$ is a complex conjugation.
In addition,
if $\det Y=1$, the real torus $T_Y$ is said to be ${\sf special}$.

\vskip 0.3cm
\noindent{\bf Example 3.2.} Let $Q=\,\begin{pmatrix} \sqrt{2} & \ \sqrt{3} \\ \sqrt{3} & -\sqrt{5}
\end{pmatrix}$ be a $2\times 2$ symmetric real matrix of signature $(1,1)$.
Then $\La_Q=\,Q\BZ^2$ is a lattice in $\BR^2$. Then the real torus $T_Q=\,\BR^2/\La_Q$ is not
polarized because the associated complex torus $\FA_Q=\,\BC^2/L_Q$ is not an abelian variety, where
$L_Q=\BZ^2+\,i\,\La_Q$ is a lattice in $\BC^2$.

\begin{definition}
Two polarized tori $T_1\,=\,\BR^n/\La_1$ and $T_2\,=\,\BR^n/\La_2$ are said to be isomorphic if the associated
polarized real abelian varieties $\FA_1\,=\,\BC^n/L_1$ and $\FA_2\,=\,\BC^n/L_2$ are isomorphic, where
$L_i\,=\,\BZ^n\,+\,i\,\La_i\ (i=1,2),$ more precisely, if there exists a linear isomorphism $\varphi:\BC^n\lrt \BC^n$ such that
\begin{eqnarray}
\varphi (L_1)&=&L_2,\\
\varphi_* (E_1)&=&E_2,\\
\varphi_* (\s_1) &=& \varphi\circ \s_1\circ \varphi^{-1}\,=\,\s_2,
\end{eqnarray}
where $E_1$ and $E_2$ are polarizations of $\FA_1$ and $\FA_2$ respectively, and $\s_1$ and $\s_2$ denotes
the real structures (in fact complex conjugations) on $\FA_1$ and $\FA_2$ respectively.
\end{definition}

\noindent{\bf Example 3.3.} Let $Y_1$ and $Y_2$ be two $n\times n$ positive definite symmetric real matrices. Then
$\La_i:=\,Y_i\,\BZ^n$ is a lattice in $\BR^n$ $(i=1,2)$. We let
$$T_i:=\,\BR^n/\La_i,\qquad i=1,2$$
be real tori of dimension $n$. Then according to Example 3.1, $T_1$ and $T_2$ are principally polarized real tori.
We see that $T_1$ is isomorphic to $T_2$ as polarized real tori if and only if there is an element $A\in GL(n,\BZ)$
such that $Y_2\,=\,A\,Y_1\,{}^tA.$

\end{section}


\vskip 10mm
\begin{section}{{\bf The moduli space of polarized real tori}}
\setcounter{equation}{0}

\vskip 3mm
 For a given fixed positive integer $n$,
we let
$${\mathbb H}_n=\,\{\,\Omega\in \BC^{(n,n)}\,|\ \Om=\,^t\Om,\ \ \ \text{Im}\,\Om>0\,\}$$
be the Siegel upper half plane of degree $n$ and let
$$Sp(n,\BR)=\{ M\in \BR^{(2n,2n)}\ \vert \ ^t\!MJ_nM= J_n\ \}$$
be the symplectic group of degree $n$, where
$$J_n=\begin{pmatrix} 0&I_n\\
                   -I_n&0\end{pmatrix}.$$
Then $Sp(n,\BR)$ acts on $\BH_n$ transitively by
\begin{equation}
M\cdot\Om=(A\Om+B)(C\Om+D)^{-1},
\end{equation} where $M=\begin{pmatrix} A&B\\
C&D\end{pmatrix}\in Sp(n,\BR)$ and $\Om\in \BH_n.$ Let
$$\G_n^{\flat}:=Sp(n,\BZ)=\left\{ \begin{pmatrix} A&B\\
C&D\end{pmatrix}\in Sp(n,\BR) \,\big| \ A,B,C,D\
\textrm{integral}\ \right\}$$ be the Siegel modular group of
degree $n$. This group acts on $\BH_n$ properly discontinuously.

\vskip 0.35cm
Let ${\mathcal A}_n
:=\G_n^{\flat}\backslash \BH_n$ be the Siegel modular variety of degree $n$, that is, the moduli space of $n$-dimensional principally polarized abelian varieties,
and let ${\mathcal M}_n$ be the the moduli space of projective curves of genus $n$. Then according to Torelli's theorem, the Jacobi mapping
\begin{equation}
T_n:{\mathcal M}_n \lrt {\mathcal A}_n
\end{equation}
defined by
\begin{equation*}
C \longmapsto J(C):= {\rm the\ Jacobian\ of}\ C
\end{equation*}
is injective. The Jacobian locus $J_n:=T_n({\mathcal M}_n)$ is a $(3n-3)$-dimensional subvariety
of ${\mathcal A}_n$ if $n\geq 2.$ We denote by $Hyp_n$ the hyperelliptic locus in ${\mathcal A}_n.$

\vskip 3mm
If $Y\in \CP$, according to Example 3.1, $T_Y=\BR^n/\La_Y$ is a principally polarized real torus of dimension $n$ and $\FA_Y=\BC^n/L_Y$ is a principally polarized abelian variety of dimension $n$. Here
$\La_Y=Y\BZ^n$ is a lattice in $\BR^n$ and $L_Y=\BZ^n + i\La_Y$ is a lattice in $\BC^n$. We denote by
$[\FA_Y]$ the isomorphism class of $\FA_Y$.

\vskip 2mm
\newcommand\FY{\mathfrak Y}
\newcommand\FJ{\mathfrak J}
The arithmetic quotient
\begin{equation*}
  \FY_n := \G_n\backslash GL(n,\BR)/O(n,\BR),\qquad \G_n:=GL(n,\BZ)/\{ \pm I_n \}
\end{equation*}
is the moduli space of principally polarized real tori of dimension $n$.

\vskip 3mm
We define
\begin{equation*}
 \FJ_{n,J}:=\left\{  Y\in \CP\,\vert\ \mathfrak A_Y\
 {\rm is\ the\ Jacobian\ of\ a\ curve}\ {\rm of\ genus}\ n,
 \ i.e., \ [\mathfrak A_Y]\in J_n \,   \right\}
\end{equation*}
and
\begin{equation*}
 \FJ_{n,H}:=\left\{  Y\in \CP\,\vert\ \mathfrak A_Y\
 {\rm is\ the\ Jacobian\ of\ a\ hyperelliptic\ curve}\
 {\rm of\ genus}\ n\,  \right\}.
\end{equation*}
We see that $\G_n$ acts on both $\FJ_{n,J}$ and $\FJ_{n,H}$ properly discontinously. So we may define
\begin{equation*}
  \mathfrak Y_{n,J}:=\G_n\backslash \FJ_{n,J}\qquad {\rm and} \qquad
  \mathfrak Y_{n,H}:=\G_n\backslash \FJ_{n,H}.
\end{equation*}
$\mathfrak Y_{n,J}$ and $\mathfrak Y_{n,H}$ are called the
${\sf Jacobian\ real\ locus}$ and the ${\sf hyperelliptic\ real\ locus}$ respectively.

\vskip 3mm
The following natural problem may be regarded as the real version of the Schottky problem.
\vskip 2mm \noindent
{\bf Problem.} Characterize the Jacobian real locus $\mathfrak Y_{n,J}$.

\vskip 3mm
For any positive integer $n\in \BZ^+$, let
\begin{equation*}
  G_n=GL(n,\BR),\quad K_n=O(n,\BR)\quad {\rm and}\quad \G_n=GL(n,\BZ)/\{ \pm I_n \}.
\end{equation*}
For any $m,n\in \BZ^+$ with $m< n$, we define
\begin{equation*}
  \xi_{m,n}:G_m\lrt G_n
\end{equation*}
by
\begin{equation}
  \xi_{m,n}(A):=
  \begin{pmatrix}
    A & 0 \\
    0 & I_{n-m}
  \end{pmatrix},\qquad A\in G_m.
\end{equation}
We let
\begin{equation*}
  G_\infty:=\varinjlim_n G_n,\qquad K_\infty:=\varinjlim_n K_n
  \quad {\rm and}\quad \G_\infty:=\varinjlim_n \G_n
\end{equation*}
be the inductive limits of the directed systems $(G_n,\xi_{m,n}),
\ (K_n,\xi_{m,n})$ and $(\G_n,\xi_{m,n})$ respectively.

\vskip 3mm
For any two positive integers $m,n\in\BZ^+$ with $m<n$, we embed
$\mathscr P_m$ into $\mathscr P_n$ as follows:
\begin{equation*}
 \psi_{m,n}:\mathscr P_m \lrt \mathscr P_n, \qquad
 Y \mapsto \begin{pmatrix}
              Y & 0 \\
              0 & I_{n-m}
            \end{pmatrix},\qquad Y\in \mathscr P_m.
\end{equation*}
We let
\begin{equation*}
  \mathscr P_\infty= \varinjlim_n \mathscr P_n
\end{equation*}
be the inductive limit of the directed system $(\mathscr P_n,\psi_{m,n}).$
We can show that
\begin{equation*}
  \mathscr P_\infty= G_\infty/ K_\infty.
\end{equation*}

\vskip 3mm
Let $\mathfrak Y_n^{S}$ be the Satake compactification of
$\mathfrak Y_n$ (cf.\,Theorem 3 in \cite[pp.\,62-65]{Gr3} or Theorem 5.2 in this article).
We denote by $\mathfrak Y_{n,J}^{S}$ (resp.\,
$\mathfrak Y_{n,H}^{S}$) the Satake compactification of
$\mathfrak Y_{n,J}$ (resp.\,$\mathfrak Y_{n,H}$). We can show that
$\mathfrak Y_{n,J}^{S}$ (resp.\,$\mathfrak Y_{n,H}^{S}$) is the closure of $\mathfrak Y_{n,J}$ (resp.\,$\mathfrak Y_{n,H}$) inside
$\mathfrak Y_n^{S}$. We have the following sequences

\begin{equation*}
  \mathfrak Y_1^{S}\lrt \mathfrak Y_2^{S}\lrt
  \mathfrak Y_3^{S}\lrt \cdots,
\end{equation*}

\begin{equation*}
  \mathfrak Y_{1,J}^{S}\lrt \mathfrak Y_{2,J}^{S}\lrt
  \mathfrak Y_{3,J}^{S}\lrt \cdots
\end{equation*}
and
\begin{equation*}
  \mathfrak Y_{1,H}^{S}\lrt \mathfrak Y_{2,H}^{S}\lrt
  \mathfrak Y_{3,H}^{S}\lrt \cdots.
\end{equation*}

\vskip 5mm\noindent
We put
\begin{equation*}
 \mathfrak Y_{\infty}^{S}:=\varinjlim_n \mathfrak Y_n^S,\qquad
 \mathfrak Y_{\infty,J}^{S}:=\varinjlim_n \mathfrak Y_{n,J}^S\quad  {\rm and}\quad
 \mathfrak Y_{\infty,H}^{S}:=\varinjlim_n \mathfrak Y_{n,H}^S.
\end{equation*}

\end{section}


\vskip 10mm
\begin{section}{{\bf Automorphic forms for $GL(n,\BR)$}}
\setcounter{equation}{0}

\vskip 3mm
Let
\begin{equation*}
  \SP:=\left\{ Y\in \BR^{(n,n)}\,|\ Y=\,^tY >0,\ \det (Y)=1\ \right\}.
\end{equation*}
be a symmetric space associated to $SL(n,\BR)$. Indeed, $SL(n,\BR)$ acts on $\SP$ transitively by
\begin{equation}
  g\cdot Y= gY\,{}^tg,\qquad g\in SL(n,\BR),\ Y\in \SP.
\end{equation}
Thus $\SP$ is a smooth manifold diffeomorphic to the symmetric space $SL(n,\BR)/SO(n,\BR)$ through the bijective map
\begin{equation*}
  SL(n,\BR)/SO(n,\BR) \lrt \SP,\qquad g\,SO(n,\BR) \mapsto g\,{}^tg,\quad g\in SL(n,\BR).
\end{equation*}

\vskip 3mm
For $Y\in\SP$, we have a partial Iwasawa decomposition
\begin{equation}
  Y= \begin{pmatrix}
       v^{-1} & 0 \\
       0 & v^{1/(n-1)}W
      \end{pmatrix}
       \left[ \begin{pmatrix}
                1 & ^tx \\
                0 & I_{n-1}
              \end{pmatrix} \right] =\begin{pmatrix}
                v^{-1} & v^{-1}\,^tx \\
                v^{-1}x & v^{-1}x\,^tx + v^{1/(n-1)}W
              \end{pmatrix}
\end{equation}
where $v>0,\ x\in \BR^{(n-1,1)}$ and $W\in \mathfrak{P}_{n-1}.$ From now on, for brevity, we
write $Y=[v,x,W]$ instead of the decomposition (5.2). In these coordinates $Y=[v,x,W]$,
\begin{equation*}
  ds_Y^2 =\frac{n}{n-1}\, v^{-2} dv^2 + 2\,v^{-n/(n-1)} W^{-1}[dx] + ds_W^2
\end{equation*}
is a $SL(n,\BR)$-invariant metric on $\SP$, where $dx=\,^t(dx_1,\cdots, dx_{n-1})$ and
$ds_W^2$ is a $SL(n-1,\BR)$-invariant metric on $\mathfrak P_{n-1}$.
The Laplace operator $\Delta_n$ of $(\SP, ds_Y^2)$ is given by
\begin{equation*}
   \Delta_n=\frac{n-1}{n}\, v^{2} {{\partial^2}\over {\partial v^2}} - \frac{1}{n}
   {{\partial}\over {\partial v}} + v^{n/(n-1)}\, W \left[ {{\partial}\over {\partial x}} \right]+\Delta_{n-1}
\end{equation*}
inductively, where if $x=\,^t(x_1,\cdots,x_{n-1})\in \BR^{(n-1,1)}$,
$${{\partial}\over {\partial x}}=\,{}^{{}^{{}^{{}^\text{\scriptsize $t$}}}}\!\!\!
\left( {{\partial}\over {\partial x_1}},\cdots,
{{\partial}\over {\partial x_{n-1}}}  \right)$$
and $\Delta_{n-1}$ is the Laplace operator of $(\mathfrak P_{n-1}, ds^2_W)$.
\begin{equation*}
d\mu_n=v^{-(n+2)/2}\, dv\, dx\, d\mu_{n-1}
\end{equation*}
is a $SL(n,\BR)$-invariant volume element on $\SP$, where $dx=dx_1\cdots dx_{n-1}$ and
$d\mu_{n-1}$ is a $SL(n-1,\BR)$-invariant volume element on $\mathfrak P_{n-1}$.

\vskip 3mm
Following earlier work of Minkowski, Siegel \cite{Sie} showed that the volume of the fundamental domain
$SL(n,\BZ)\backslash \SP$ is given as follows\,:
\begin{equation}
  {\rm Vol}(SL(n,\BZ)\backslash \SP)
  =\int_{SL(n,\BZ)\backslash \SP}d\mu_n=
  n\,2^{n-1}\prod_{k=2}^{n}{{\zeta (k)}\over {{\rm Vol}(S^{k-1})}},
\end{equation}
where
\begin{equation*}
{\rm Vol}(S^{k-1})={{2\,(\sqrt{\pi})^k}\over {\Gamma(k/2)}}
\end{equation*}
denotes the volume of the $(k-1)$-dimensional sphere $S^{k-1}$, $\Gamma (x)$ denotes the usual Gamma function and $\zeta (k)=\sum_{m=1}^{\infty} m^{-k}$ denotes the Riemann zeta function. The proof of (5.3) can be found in \cite[pp.\,27-37]{G} and \cite{Sie}.

\vskip 3mm
If we repeat this partial decomposition process for $W$, we get the Iwasawa decomposition
\begin{equation}
  Y=y^{-1} {\rm{diag}}
  \left( 1,y_1^2,(y_1y_2)^2,\cdots,(y_1y_2\cdots y_{n-1})^2 \right)
  \left[ \begin{pmatrix}
           1 & x_{12} & \cdots & x_{1n} \\
           0 & 1 & \cdots & x_{2n}\\
           0 & 0 & \ddots &  \vdots\\
           0 & 0 &   0  & 1
         \end{pmatrix} \right],
\end{equation}
where $y>0,\ y_j\in \BR \,(1\leq j\leq n-1)$ and $x_{ij}\in\BR (1\leq i < j\leq n).$ Here
$y=y_1^{2(n-1)}\cdots y_{n-1}^2$ and ${\rm{diag}} (a_1,\cdots,a_n)$ denotes the $n\times n$ diagonal matrix with diagonal entries $a_1,\cdots,a_n$. In this case we denote $Y=(y_1,\cdots,y_{n-1},x_{12},\cdots,x_{n-1,n}).$

\vskip 5mm
Define $\Gamma_n=GL(n,\BZ)/\{ \pm I_n \}.$ We observe that $\G_n=SL(n,\BZ)/\{ \pm I_n \}$ if $n$ is even, and $\G_n=SL(n,\BZ)$ if $n$ is odd.
An \textsf{automorphic form} for $\G_n$ is defined to be a real analytic function $f\!:\!\SP\lrt \BC$ satisfying the following conditions (AF1)--(AF3)\,:
\vskip 2mm
(AF1) $f$ is an eigenfunction for all
$GL(n,\BR)$-invariant\ differential\ operators on $\SP$.
\vskip 2mm
(AF2) \ \ $f(\g Y\,{}^t\g)=f(Y)\quad {\rm for\ all}\ \g\in SL(n,\BR)
\ {\rm and}\ Y\in \SP.$
\vskip 2mm
(AF3)\ There\ exist\ a constant  $C>0$ and $s\in \BC^{n-1}$
with  $s=(s_1,\cdots,s_{n-1})$   \\
\indent \ \ \ \ \ \ \ \ \ such that
$|f(Y)| \leq C\,|p_{-s}(Y)|$ as the upper left determinants
$\det Y_j\lrt \infty,$ \\
\indent \ \ \ \ \ \ \ \ \ $j=1,2,\cdots, n$, where
\begin{equation*}
  p_{-s}(Y):=\prod_{j=1}^{n-1} (\det Y_j)^{-s_j}
\end{equation*}
\indent \ \ \ \ \ \ \ \ \
is the Selberg's power function\,(cf.\,\cite{S1, T}).

\vskip 3mm
We denote by $A(\G_n)$ the space of all automorphic forms for $\G_n.$ A \textsf{cusp form} $f\in A(\G_n)$ is an automorphic form for $\G_n$ satisfying the following conditions\,:
\begin{equation*}
\int_{X\in (\BR/\BZ)^{(j,n-j)}}
f \left( Y\left[ \begin{pmatrix}
                   I_j & X \\
                   0 & I_{n-j}
                 \end{pmatrix}\right]\right)dX=0,
                 \quad 1\leq j\leq n-1.
\end{equation*}
We denote by $A_0(\G_n)$ the space of all cusp forms for $\G_n.$

\vskip 3mm
For $s=(s_1,\cdots,s_{n-1})\in \BC^{n-1},$ we now consider the following Eisenstein series
\begin{equation}
  E_n(s,Y):=\sum_{\gamma\in \G_n/\G_{\star}} p_{-s} ( Y[\gamma]),\qquad  Y\in\SP,
\end{equation}
where $\G_{\star}$ is a subgroup of $\G_n$ consisting of all upper triangular matrices.
$E_n(s,Y)$ is a special type of the more general Eisenstein series introduced by Atle Selberg
\cite{S2}. It is known that the series converges for ${\rm Re} (s_j)>1,\ j=1,2,\cdots, n-1,$ and has analytic continuation for each of the variables $s_1,\cdots,s_{n-1}$ (cf.\,\cite{S2, T}). It is seen that $E_n(s,Y)$ is a common eigenfunction of all invariant differential operators in $\BD (\CP)$.
Its corresponding eigenvalue of the Laplace operator $\Delta_n$ is given by
\begin{eqnarray*}
  \lambda &=& \frac{n-1}{n} (s_1+\xi_1)\left( s_1-1+\xi_1 - \frac{1}{n-1} \right) \\
   & & + \frac{n-2}{n-1} (s_2+\xi_2)\left( s_2-1+\xi_2 - \frac{1}{n-2} \right)+\cdots +
   \frac{1}{2} s_{n-1}(s_{n-1}-2),
\end{eqnarray*}
where $\xi_{n-1}=0$ and
\begin{equation}
  \xi_j=\frac{1}{n-j} \sum_{k=j+1}^{n-1} (n-k) s_k,\qquad j=1,2,\cdots,n-2.
\end{equation}

\vskip 3mm
Let $f\in A(\G_n)$ be an automorphic form. Since $f(Y)$ is invariant under the action of
the subgroup
$\left\{ \begin{pmatrix}
   1 & a \\
   0 & I_{n-1}
 \end{pmatrix} \ \Big|\ a\in \BZ^{(1,n-1)}\,\right\}$ of $\G_n$, we have the Fourier expansion
\begin{equation}
  f(Y)=\sum_{N\in \BZ^{(n-1,1)}} a_N (v,W) e^{2\pi i\, ^txN},
\end{equation}
where
$Y=[v,x,W]\in \SP$ and
\begin{equation*}
  a_N (v,W)=\int_{0}^{1}\int_{0}^{1}\cdots \int_{0}^{1} f([v,x,W])\, e^{-2\pi i\, ^txN} dx.
\end{equation*}

\vskip 2mm
For $s\in \BC^{n-1}$ and $A,B\in \SP$, we define the $K$-Bessel function
\begin{equation}
K_n (s\,|\,A,B):=\int_{\CP} p_s(Y) \,e^{Tr (AY+BY^{-1})}d\mu_n,
\end{equation}
where $d\mu_n$ is the $GL(n,\BR)$-invariant volume element on $\CP$\,(see (2.2)).

\vskip 3mm
Let $f\in A(\G_n)$ be an automorphic form for $\G_n$. Thus we have $n$ differential equations
$D_j f=\lambda_j f\,(1\leq j\leq n).$ Here $D_1,\cdots,D_n$ are $GL(n,\BR)$-invariant differential operators
defined by (2.3). We can find $s=(s_j)\in \BC^{n-1}$ satisfying the various relations determined by the
$\lambda_j\,(1\leq j\leq n)$. D. Grenier \cite[Theorem 1, pp.\,469-471]{Gr2} proved that $f$ has the following Fourier expansion
\begin{eqnarray}
  f(Y) &=& f([v,x,W]) \nonumber\\
   &=& a_0(v,W)+\sum_{0\neq m\in \BZ^{n-1}} \sum_{\gamma\in \G_{n-1}/P} a_m (v,W) v^{(n-1)/2}\\
   & & \times\, K_{n-1}({\hat s}\,|\,v^{1/(n-1)}\,^t\gamma W\gamma,\pi^2 v\, m\,^tm)\,e^{2\pi i\, ^tx\gamma m},\nonumber
\end{eqnarray}
where ${\hat s}=\left( s_1-\frac{1}{2},s_2,\cdots,s_{n-1} \right)$ and $P$ denotes the parabolic subgroup of $\G_{n-1}$ consisting of the form
$\begin{pmatrix}
  \pm 1 & b \\
  0 & d
\end{pmatrix}$ with $d\in \G_{n-2}.$

\vskip 3mm
D. Grenier \cite{Gr1, Gr3} found a fundamental domain ${\mathfrak F}_n$ for $\G_n$ in $\SP$.
The fundamental domain ${\mathfrak F}_n$ is precisely the set of all $Y=[v,x,W]\in \SP$ satisfying
the following conditions (F1)-(F3)\,:
\vskip 1mm\noindent
(F1) $(a+\,^txc)^2+v^{n/(n-1)} W[c]\geq 1$ for all
$\begin{pmatrix}
  a & ^t b \\
  c & d
\end{pmatrix}\in \G_n$ with $a\in \BZ,\ b,c\in \BZ^{(n-1,1)}$ \\
\indent \indent and $d\in \BZ^{(n-1,n-1)}.$

\vskip 2mm\noindent
(F2) $W\in {\mathfrak F}_{n-1}.$

\vskip 2mm\noindent
(F3)\ $0\leq x_1 \leq \frac{1}{2},\ \, |x_j|\leq 2$ for $2\leq j\leq n-2.$ Here
$x=\,^t (x_1,\cdots,x_{n-1})\in \BR^{(n-1,1)}.$

\vskip 3mm
For a positive real number $t>0$, we define the $\textsf{Siegel\ set} \ {\mathscr S}_{t,1/2}$ by
\begin{equation*}
  {\mathscr S}_{t,1/2}:=\left\{ Y\in \SP\,|\ y_i\geq t^{-1/2}\,(1\leq i\leq n-1),\ |x_{ij}|\leq
  \frac{1}{2}\,(1\leq i < j \leq n)\,\right\}.
\end{equation*}
Here we used the coordinates on $\SP$ given in Formula (5.4).

\vskip 2mm
Grenier \cite{Gr3} proved the following theorems\,:
\begin{theorem}
  Let
\begin{equation*}
 {\mathfrak F}_n^{\sharp}:=\bigcup_{\gamma\in D_n} {\mathfrak F}_n [\gamma]\subset \SP,
\end{equation*}
where $D_n$ is the subgroup of $\G_n$ consisting of diagonal matrices
${\rm diag}(\pm 1,\cdots,\pm 1).$ Then
\begin{equation*}
{\mathscr S}_{1,1/2} \subset {\mathfrak F}_n^{\sharp} \subset {\mathscr S}_{4/3,1/2}.
\end{equation*}
\end{theorem}
\noindent
{\it Proof.} See Theorem 1 in \cite[pp.\,58-59]{Gr3}. \hfill $\square$

\begin{theorem}
  Let
\begin{equation*}
 {\mathfrak F}_n^*:={\mathfrak F}_n\cup {\mathfrak F}_{n-1} \cup \cdots \cup {\mathfrak F}_2
 \cup {\mathfrak F}_1
\end{equation*}
and
\begin{equation*}
V_n^*:=V_n \cup V_{n-1} \cup \cdots \cup V_1 \cup V_0,\qquad V_n:=\G_n\backslash \SP.
\end{equation*}
Then ${\mathfrak F}_n^*$ is a compact Hausdorff space whose topology is induced by the closure of
${\mathfrak F}_n$ in $\SP.$ And $V_n^*$ is a compact Hausdorff space called
the {\bf Satake\ compactification} of
$V_n$.
\end{theorem}
\noindent
{\it Proof.} See Theorem 3 in \cite[pp.\,62-65]{Gr3}. \hfill $\square$

\end{section}

\vskip 10mm


\begin{section}{{\bf Stable automorphic forms for the general linear group}}
\setcounter{equation}{0}

\vskip 3mm
In this section, we introduce the concept of the stability of automorphic forms for the general linear group using the Grenier operator, and relate the stability of automorphic forms to the moduli space of principally polarized real tori and the Jacobian real locus.

\begin{definition}
Let $f\in A(\G_n)$ be an automorphic form for $\G_n$ with eigenvalues determined by $s=(s_1,\cdots,s_{n-1})\in \BC^{(n-1)}$.
We set
\begin{equation*}
  \xi_1={\frac{1}{n-1}} \sum_{n=2}^{n-1} (n-k)s_k  \qquad ({\rm cf}.\ {\rm Formula}\ (5.6)).
\end{equation*}
We define formally, for any $f\in A(\G_n)$,
\begin{equation}
  \mathfrak L_n f (W):=\lim_{v\lrt \infty} v^{-s_1-\xi_1}f(Y),
  \quad v>0,\ W\in {\mathfrak P}_{n-1},\ Y\in \SP,
\end{equation}
where $Y,\, v,\, W$ are determined by the unique partial Iwasawa decomposition of $Y$
given by
\begin{equation*}
  Y=\begin{pmatrix}
      1 & 0 \\
      x & I_{n-1}
    \end{pmatrix}
    \begin{pmatrix}
      v^{-1} & 0 \\
      0 & v^{\frac{1}{n-1}}W
    \end{pmatrix}
    \begin{pmatrix}
      1 & {}^tx \\
      0 & I_{n-1}
    \end{pmatrix}\quad ({\rm see}\ (5.2)).
\end{equation*}
\end{definition}

D. Grenier \cite{Gr2} defined the formula (6.1) and proved the following result.

\begin{theorem}
If $f\in A(\G_n)$, then $\mathfrak L_n f\in A(\G_{n-1}).$ Thus
$\mathfrak L_n$ is a linear mapping of $A(\G_n)$ into $A(\G_{n-1})$.
Moreover if $f\in A_0(\G_n)$ is a cusp form, then $\mathfrak L_n f=0.$ In general, $\ker \mathfrak L_n\neq A_0(\G_n).$
\end{theorem}
\noindent
{\it Proof.} The detailed proof can be found in \cite[Theorem 2, pp.\,472-473]{Gr2}.
For the convenience of the reader, we sketch Grenier's proof. We consider the Fourier expansion (5.9)
of $f$. Using the properties of the $K$-Bessel functions (cf. Formula (5.8)) and the Selberg power functions, Grenier showed that as $v\lrt \infty$,
\begin{equation*}
  f(Y)=f([v,x,W]) \thicksim a_0 (v,W).
\end{equation*}
Therefore it suffices to show that
$$\mathfrak L_n f (W):=\lim_{v\lrt \infty} v^{-s_1-\xi_1} a_0 (v,W)$$
satisfies the properties (AF1),\,(AF2) and (AF3).
\vskip 2mm\noindent
(1) Since $a_0 (v,W)$ contains no $x=\,^t(x_1,\cdots,x_{n-1})$ terms, each $D_j$ reduces to
$D_j=D_{j,v} + D_{j,W}\, (1\leq j\leq n)$, where $D_{j,v}$ operates on $v$ alone and $D_{j,W}$
is the corresponding $GL(n-1,\BR)$-invariant differential operator on $\mathfrak P_{n-1}$. Thus
$\mathfrak L_n f (W)$ is a joint eigenfunction of $\BD (\mathfrak P_{n-1})$.
\vskip 2mm\noindent
(2) It is easily seen that $a_N (v,W[\gamma])=a_{\gamma N}(v,W)$ for all $\gamma\in \G_{n-1}$
(see Formula (5.7)).
Taking $N=0$, we get $a_0 (v,W[\gamma])=a_0 (v,W)$ for all $\gamma\in \G_{n-1}.$ Thus
$\mathfrak L_n f (W[\gamma])=\mathfrak L_n f (W)$ for all $\gamma\in \G_{n-1}.$
\vskip 2mm\noindent
(3) There exist a constant $C>0$ and $s\in \BC^{n-1}$ with $s=(s_1,\cdots,c_{n-1})$ such that
$| f(Y)|\leq C |p_{-s}(Y)|$ as $\det Y_j \lrt \infty\ (1\leq j\leq n)$. Since
\begin{equation*}
  a_0 (v,W)=\int_{0}^{1}\int_{0}^{1}\cdots \int_{0}^{1} f([v,x,W])dx,
\end{equation*}
we get
\begin{eqnarray*}
  |a_0 (v,W)| &\leq& \int_{0}^{1}\int_{0}^{1}\cdots \int_{0}^{1} |f([v,x,W])|\,dx \\
   &\leq& C\, \int_{0}^{1}\int_{0}^{1}\cdots \int_{0}^{1} |p_{-s}([v,x,W])|\,dx\\
   &=& C\, \int_{0}^{1}\int_{0}^{1}\cdots \int_{0}^{1} |v^{s_1+\xi_1}|\,|p_{-s'}(W)|\,dx \\
   &=& C\,|v^{s_1+\xi_1}|\,|p_{-s'}(W)|,
\end{eqnarray*}
where $s'=(s_2,\cdots,s_{n-1})$. Hence $|\mathfrak L_n f (W)|\leq C\,|p_{-s'}(W)|$.
\hfill $\square$

\begin{remark}
The reason that $\ker \mathfrak L_n\neq A_0(\G_n)$ in general is that $\mathfrak L_n$ is
only one of several such operators associated with the various maximal parabolic subgroups.
\end{remark}

\vskip 3mm
For any $m,n\in \BZ^+$ with $m<n$, we define
\begin{equation*}
  \xi_{m,n}:\G_m\lrt \G_n
\end{equation*}
by
\begin{equation*}
  \xi_{m,n}(\gamma):=
  \begin{pmatrix}
    \gamma & 0 \\
    0 & I_{n-m}
  \end{pmatrix},\qquad \gamma\in \G_m.
\end{equation*}
We let
\begin{equation*}
  \G_\infty:=\varinjlim_n \G_n
\end{equation*}
be the inductive limit of the directed system $(\G_n,\xi_{m,n}).$

\vskip 3mm

\begin{definition}
A collection $(f_n)_{n\geq 1}$ is said to be a {\sf stable\
automorphic form} for $\G_\infty$ if it satisfies the following conditions (6.2) and (6.3)\,:
\begin{equation}
  f_n\in A(\G_n),\quad n\geq 1
\end{equation}
and
\begin{equation}
  \mathfrak L_{n+1}f_{n+1}=f_n,\quad n\geq 1.
\end{equation}
\end{definition}

\vskip 3mm
Let
\begin{equation*}
  {\mathbb A}_\infty=A(\G_\infty):=\varprojlim_n A(\G_n)
\end{equation*}
be the inverse limit of the directed system $(A(\G_n),\mathfrak L_n)$, that is,
the space of all stable automorphic forms for $\G_\infty$.

\vskip 5mm
We propose the following problems.
\vskip 3mm\noindent
{\bf Problem 6.1.} Discuss the injectivity, the surjectivity and the bijectivity of $\mathfrak L_n.$
\vskip 2mm\noindent
{\bf Problem 6.2.} Give examples of stable automorphic forms for
$\G_\infty$.
\vskip 2mm\noindent
{\bf Problem 6.3.} Investigate the structure of ${\mathbb A}_\infty$.
\begin{remark}
In the classical case of Siegel modular forms, Freitag \cite{F1} showed that the
ring structure of stable Siegel modular forms corresponding similarly to $\mathbb A_\infty$
is the polynomial ring in the theta series associated to irreducible, positive definite, unimodular even
quadratic forms.
\end{remark}

\vskip 5mm
We give an example of stable automorphic forms for $\G_\infty$.

\begin{theorem}
  Let $\{\alpha_n \}$ be the sequence such that if $\alpha_n=(s_1,s_2,\cdots,s_{n-1})\in \BC^{n-1}$, then
  $\alpha_{n-1}=(s_2,s_3,\cdots,s_{n-1})\in \BC^{n-2}$ for any $n\in \BZ^+$. Then $( E_n(\alpha_n,Y))$ is a stable automorphic form for $\G_\infty$. Here $E_n(\alpha_n,Y)$ is the Selberg Eisenstein series defined by Formula (5.5).
\end{theorem}
\noindent
{\it Proof.} We prove the above theorem using the Grenier's result \cite[p.\,472]{Gr2}.
Let
\begin{equation*}
  Y=[v,x,W]= \begin{pmatrix}
       v^{-1} & 0 \\
       0 & v^{1/(n-1)}W
      \end{pmatrix}
       \left[ \begin{pmatrix}
                1 & ^tx \\
                0 & I_{n-1}
              \end{pmatrix} \right].
\end{equation*}
Let $s=(s_1,\cdots,s_{n-1})\in \BC^{n-1}.$
For any $\gamma=\begin{pmatrix}
       a & ^tb \\
       c & D
      \end{pmatrix}\in \G_n$ with $a\in\BZ,\ b,c\in \BZ^{(n-1,1)}$ and $D\in \BZ^{(n-1,n-1)},$
we have
\begin{equation*}
  Y[\gamma]=\,^t\gamma Y \gamma=\begin{pmatrix}
       \alpha & q \\
       ^tq & R
      \end{pmatrix},
\end{equation*}
where
\begin{eqnarray*}
  \alpha &=& v^{-1}(a+\, ^tcx)^2 + v^{1/(n-1)} W[c], \\
  q &=& v^{-1}(a+\, ^tcx)(\, ^tb+\, ^txD)+ v^{1/(n-1)}\, ^tcWD,  \\
  R &=& v^{-1}(b+\, ^tDx)(\, ^tb+\, ^txD) + v^{1/(n-1)} W[D].
\end{eqnarray*}
We observe that
\begin{equation*}
  Y[\gamma]= \begin{pmatrix}
       \alpha & 0 \\
       0 & R-\alpha^{-1}\, ^tq q
      \end{pmatrix}
       \left[ \begin{pmatrix}
                1 & \alpha^{-1}q \\
                0 & I_{n-1}
              \end{pmatrix} \right].
\end{equation*}
Let $\alpha_n=(s_1,\cdots,s_{n-1})\in \BC^{n-1}.$
Set $\xi_1=\frac{1}{n-1} \sum_{k=2}^{n-1} (n-k) s_k.$ By Proposition 2.1, we see that
\begin{equation*}
  p_{-\alpha_n} (Y[\gamma])=\alpha^{-(s_1+\xi_1)} p_{-\alpha_{n-1}}((v\alpha)^{\frac{2-n}{n-1}}\, Y_*),
\end{equation*}
where $\alpha_{n-1}=(s_2,\cdots,s_{n-1})\in \BC^{n-2}$ and
\begin{equation*}
Y_*=W[(a+ {^tx}c)D-c(\,^tb+\,^txD)].
\end{equation*}
It is easily seen that
\begin{equation*}
p_{-\alpha_n}(Y[\gamma])=v^{s_1+\xi_1} (v\alpha)^{-s_1+s_3+2s_4+\cdots+(n-3)s_{n-1}} p_{-\alpha_{n-1}}(Y_*).
\end{equation*}
We note that
\begin{equation*}
v\alpha=(a+\,^txc)^2 + v^{n/(n-1)}W[c].
\end{equation*}
Thus we have
\begin{eqnarray*}
  E_n (\alpha_n,Y) &=& \sum_{\gamma\in \G_n/\G_{\infty} } p_{-\alpha_n} (Y[\gamma]) \\
   &=& v^{s_1+\xi_1} \sum_{\gamma\in \G_n/\G_{\star} }
       (v\alpha)^{-s_1+s_3+2s_4+\cdots+(n-3)s_{n-1}} p_{-\alpha_{n-1}} (Y_*)\\
   &=& v^{s_1+\xi_1} \sum_{\gamma\in \G_n/\G_{\star} }
       \left\{ (a+\, ^tcx)^2 + v^{n/(n-1)} W[c] \right\}^{-s_1+s_3+2s_4+\cdots+(n-3)s_{n-1}}  \\
   & & \ \ \times\, p_{-\alpha_{n-1}}(W[(a+ {^tx}c)D-c(\,^tb+\,^txD)]),
\end{eqnarray*}
where $\G_{\star}$ is a subgroup of $\G_n$ consisting of all upper triangular matrices.
As $v\lrt \infty,$ if $s$ is chosen to make the exponent negative, all the terms with $c\neq 0$
approach to zero. If $c=0$, then $\gamma=\begin{pmatrix} a & ^tb \\ 0 & D\end{pmatrix}\in \G_n$ and so $a=\pm 1.$ Thus as $v\lrt \infty,$ we have
\begin{equation*}
E_n (\alpha_n,Y)\, \thicksim\, v^{s_1+\xi_1} \sum_{\gamma\in \G_{n-1}/\G_{\diamond} } p_{-\alpha_{n-1}}(W[\gamma])= v^{s_1+\xi_1} E_{n-1}(\alpha_{n-1},W),
\end{equation*}
where $\G_{\diamond}$ is a subgroup of $\G_{n-1}$ consisting of all upper triangular matrices.
Therefore
\begin{equation*}
\mathfrak L_n E_n (\alpha_n,Y)=\lim_{v\lrt\infty} v^{-(s_1+\xi_1)} E_n (\alpha_n,Y)=E_{n-1}(\alpha_{n-1},W).
\end{equation*}
Hence $( E_n(\alpha_n,Y))$ is a stable automorphic form for $\G_\infty$.
\hfill $\square$

\vskip 5mm
Let
\begin{equation*}
  G_n=SL(n,\BR),\quad K_n=SO(n,\BR)\quad {\rm and}\quad \G_n=GL(n,\BZ)/\{ \pm I_n\}.
\end{equation*}
We observe that $\G_n=SL(n,\BZ)/\{ \pm I_n\}$ if $n$ is even, and
$\G_n=SL(n,\BZ)$ if $n$ is odd.

\vskip 5mm
Let
\begin{equation}
  \mathfrak X_n:=\G_n\backslash \SP=\G_n\backslash G_n/ K_n
\end{equation}
be the moduli space of ${\sf special}$ principally polarized real tori of dimension $n$.

\vskip 5mm
For any $m,n\in \BZ^+,$ we define
\begin{equation*}
  \xi_{m,n}:G_m\lrt G_n
\end{equation*}
by
\begin{equation}
  \xi_{m,n}(A):=
  \begin{pmatrix}
    A & 0 \\
    0 & I_{n-m}
  \end{pmatrix},\qquad A\in G_m.
\end{equation}
We let
\begin{equation*}
  G_\infty:=\varinjlim_n G_n,\qquad K_\infty:=\varinjlim_n K_n
  \quad {\rm and}\quad \G_\infty:=\varinjlim_n \G_n
\end{equation*}
be the inductive limits of the directed systems $(G_n,\xi_{m,n}),
\ (K_n,\xi_{m,n})$ and $(\G_n,\xi_{m,n})$ respectively.

\vskip 3mm
We recall  the Jacobian locus $J_n$ (resp.\, the hyperelliptic locus ${\rm Hyp}_n$) in the Siegel modular variety $\mathcal A_n$ (see section 4).
We define
\begin{equation*}
 {\mathscr J}_{n,J}:=\left\{  Y\in \SP\,\vert\ \mathfrak A_Y\
 {\rm is\ the\ Jacobian\ of\ a\ curve}\ {\rm of\ genus}\ n,
 \ i.e., \ [\mathfrak A_Y]\in J_n \,   \right\}
\end{equation*}
and
\begin{equation*}
 {\mathscr J}_{n,H}:=\left\{  Y\in \SP\,\vert\ \mathfrak A_Y\
 {\rm is\ the\ Jacobian\ of\ a\ hyperelliptic\ curve}\
 {\rm of\ genus}\ n\,  \right\}.
\end{equation*}
See Example 3.1 for the definition of $\mathfrak A_Y.$
We see that $\G_n$ acts on both ${\mathscr J}_{n,J}$ and ${\mathscr J}_{n,H}$ properly discontinuously. So we may define
\begin{equation*}
  \mathfrak X_{n,J}:=\G_n\backslash {\mathscr J}_{n,J}\qquad {\rm and} \qquad
  \mathfrak X_{n,H}:=\G_n\backslash {\mathscr J}_{n,H}.
\end{equation*}
$\mathfrak X_{n,J}$ and $\mathfrak X_{n,H}$ are defined over the real numbers. $\mathfrak X_{n,J}$ and $\mathfrak X_{n,H}$ are called the
${\sf Jacobian\ real\ locus}$ and the ${\sf hyperelliptic\ real\ locus}$ respectively.

\vskip 5mm\noindent
{\bf Problem\,6.1.} Characterize the Jacobian real locus $\mathfrak X_{n,J}$. This problem may be the real version of the Schottky problem.

\vskip 3mm
Let $\mathfrak X_n^{S}$ be the Satake compactification of
$\mathfrak X_n.$ We denote by $\mathfrak X_{n,J}^{S}$ (resp.\,
$\mathfrak X_{n,H}^{S}$) the Satake compactification of
$\mathfrak X_{n,J}$ (resp.\,$\mathfrak X_{n,H}$). We can show that
$\mathfrak X_{n,J}^{S}$ (resp.\,$\mathfrak X_{n,H}^{S}$) is the closure of $\mathfrak X_{n,J}$ (resp.\,$\mathfrak X_{n,H}$) inside
$\mathfrak X_n^{S}$. We have the following sequences

\begin{equation*}
  \mathfrak X_1^{S}\lrt \mathfrak X_2^{S}\lrt
  \mathfrak X_3^{S}\lrt \cdots,
\end{equation*}

\begin{equation*}
  \mathfrak X_{1,J}^{S}\lrt \mathfrak X_{2,J}^{S}\lrt
  \mathfrak X_{3,J}^{S}\lrt \cdots
\end{equation*}
and
\begin{equation*}
  \mathfrak X_{1,H}^{S}\lrt \mathfrak X_{2,H}^{S}\lrt
  \mathfrak X_{3,H}^{S}\lrt \cdots.
\end{equation*}
\vskip 3mm
As far as the author knows, nobody proved or disproved so far that the Satake compactifications
$\mathfrak X_j^{S}\,(j\geq 1)$ are
projective and also normal. We propose the following problems.
\vskip 3mm\noindent
{\bf Problem\,6.2.} Are $\mathfrak X_j^{S}\,(j= 1,2,3,\cdots)$ projective and normal ?
\vskip 3mm\noindent
{\bf Problem\,6.3.} Discuss the injectivity and the surjectivity of the maps
$\mathfrak X_j^{S}\lrt \mathfrak X_{j+1}^{S}\,(j=1,2,3,\cdots)$.

\begin{remark}
For the classical Satake compactification
${\mathcal A}_n^S:=\overline{\G_n\backslash \BH_n}\,(n\geq 1)$, Baily \cite{1B} proved that
each ${\mathcal A}_n^S\,(n\geq 1)$ is a projective and normal subvariety.
For the more detailed discussion of this subject, we refer to Freitag's book
\cite[pp.\,111--124]{F2}.
\end{remark}

\vskip 3mm\noindent
We put
\begin{equation*}
 \mathfrak X_{\infty}^{S}:=\varinjlim_n \mathfrak X_n^S,\qquad
 \mathfrak X_{\infty,J}^{S}:=\varinjlim_n \mathfrak X_{n,J}^S\quad  {\rm and}\quad
 \mathfrak X_{\infty,H}^{S}:=\varinjlim_n \mathfrak X_{n,H}^S.
\end{equation*}

\vskip 3mm \noindent
For any two positive integers $m,n\in\BZ^+$ with $m<n$, we embed
$\mathfrak P_m$ into $\mathfrak P_n $ as follows:
\begin{equation*}
 \psi_{m,n}:\mathfrak P_m \lrt \mathfrak P_n, \qquad
 Y \mapsto \begin{pmatrix}
              Y & 0 \\
              0 & I_{n-m}
            \end{pmatrix},\qquad Y\in X_m.
\end{equation*}
We let
\begin{equation*}
  \mathfrak P_\infty= \varinjlim_n \mathfrak P_n
\end{equation*}
be the inductive limit of the directed system $(\mathfrak P_n,\psi_{m,n}).$
We can show that
\begin{equation*}
  \mathfrak P_\infty= G_\infty/ K_\infty.
\end{equation*}

\vskip 3mm
Now we have the Grenier operator
\begin{equation*}
  \mathfrak L_n: A(\G_n)\lrt A(\G_{n-1})
\end{equation*}
defined by the formula (6.1).

\begin{definition}
An automorphic form $f\in A(\G_n)$ is said to be a ${\sf Grenier\!-\!Schottky}$
${\sf automorphic\ form}$ for the Jacobian real locus (resp.\,the hyperelliptic real locus)
if it vanishes along
$\mathfrak X_{n,J}$ (resp.\,$\mathfrak X_{n,H}$).
A collection $(f_n)_{n\geq 1}$ is called a ${\sf stable\ Grenier\!-\!Schottky}$
${\sf automorphic\ form}$ for the Jacobian real locus (resp.\,the hyperelliptic real locus)
if it satisfies the following conditions
{\rm (SGS1)} and {\rm (SGS2)}\,:
\vskip 2mm
{\rm (SGS1)} \ \ $f_n$ is a Grenier-Schottky automorphic form
for the Jacobian real locus \\
\indent\indent \ \ \ \ \ \ \ \ \,(resp.\,the hyperelliptic real locus)
for each $n\geq 1.$
\vskip 2mm
{\rm (SGS2)}\ \  $\mathfrak L_n f_n=f_{n-1}$ for all $n> 1.$
\end{definition}

\vskip 3mm
The following natural question arises\,:
\vskip 2mm\noindent
{\bf Question 6.1.} Are there stable Grenier-Schottky automorphic forms for the Jacobian real locus (resp.\,the hyperelliptic real locus) ?

\begin{remark}
In the classical case for the Jacobian locus, Codogni and Shepherd-Barron \cite{CS} showed that
there do not exist stable Schottky-Siegel modular forms. In the classical case for the hyperelliptic
locus, Codogni \cite{Cod} showed that there exist nontrivial stable Schottky-Siegel modular forms.
\end{remark}

\end{section}

\end{document}